\documentclass{article}
\usepackage{amsthm}
\usepackage{amsmath}
\usepackage{amssymb}
\usepackage{mathrsfs}
\usepackage{tikz}
\usepackage{graphicx}
\usepackage{relsize}
\usepackage{geometry}
\usepackage{enumitem}
\usepackage{xcolor}
\usepackage{booktabs}
\usepackage{authblk}
\usepackage{float}

\usetikzlibrary{arrows.meta}
\newtheorem{theorem}{Theorem}
\newtheorem{proposition}{Proposition}
\newtheorem{lemma}{Lemma}
\newtheorem{corollary}{Corollary}
\theoremstyle{definition}
\newtheorem{definition}{Definition}
\newtheorem{question}{Question}
\newtheorem{remark}{Remark}
\newtheorem{example}{Example}

\title{Free subgroups in weighted Leavitt Path Algebras}

\author{Huynh Viet Khanh}
\affil{Department of Mathematics and Informatics, HCMC University of Education, Ho Chi Minh City, Vietnam, \texttt{khanhhv@hcmue.edu.vn}}
\date{ }

\begin{document}

\maketitle

\begin{abstract}
We study unit groups of weighted Leavitt path algebras. Let $K$ be a field
of characteristic $0$ and let $(E,\omega)$ be a finite connected weighted
graph. We prove that $L_K(E,\omega)^\times$ is abelian if and only if
$L_K(E,\omega)$ is a domain. Equivalently, $L_K(E,\omega)^\times$ contains
no non-cyclic free subgroup if and only if $L_K(E,\omega)$ is a domain.
\end{abstract}

\section{Introduction and Preliminaries}\label{sec:intro}
Let $K$ be a field and let $(E,\omega)$ be a finite connected weighted graph. This paper studies the unit group $L_K(E,\omega)^\times$ of the weighted Leavitt path algebra $L_K(E,\omega)$. We prove that $L_K(E,\omega)^\times$ is abelian if and only if $L_K(E,\omega)$ is a domain. By the domain classification of Hazrat and Preusser given in \cite{HaP}, this occurs precisely in three cases: the isolated vertex, the unweighted rose with one petal, and the LV-rose case. We also prove that, if $\operatorname{char}K=0$, then $L_K(E,\omega)^\times$ contains a non-cyclic free subgroup if and only if $L_K(E,\omega)$ is not a domain.

The existence of non-cyclic free subgroups has long been a central topic of interest in algebra, with one of the most influential starting points being the seminal work of Tits \cite{T}. In the linear setting, this phenomenon is described in the Tits alternative for finitely generated linear groups.
For division rings, Lichtman conjectured in \cite{L} that $D^*$ contains a non-cyclic free subgroup whenever $D$ is noncommutative. Gon\c{c}alves and Mandel later proposed in \cite{GM} the corresponding problem for noncentral subnormal subgroups of $D^*$. These questions remain open in this generality, although many special cases are known; see \cite{HK} and the references there.

For ordinary Leavitt path algebras, Hai and Khanh proved in \cite{HK} that if $K$ has characteristic $0$ and $L_K(E)$ is unital and noncommutative, then $L_K(E)^\times$ contains a non-cyclic free subgroup. One aim of the present paper is to prove the corresponding statement for weighted Leavitt path algebras.

Weighted Leavitt path algebras were introduced by Hazrat in \cite{Ha}. They include ordinary Leavitt path algebras as the case in which all weights are equal to $1$, and they recover Leavitt's algebras $L_K(m,n)$ from the weighted rose with one vertex and $n$ loops, all of weight $m$. Hazrat and Preusser developed a normal form theory for these algebras and classified those which are domains in \cite{HaP}. Preusser computed their Gelfand--Kirillov dimension in \cite{Pre3}, and Hazrat, Preusser and Shchegolev studied their representation theory by means of representation graphs in \cite{HPS}. For ordinary Leavitt path algebras, see Abrams, Ara, and Siles Molina \cite{AAS}; for a recent overview of weighted Leavitt path algebras, see Preusser \cite{Pre4}.

The paper is organized as follows. Section~2 treats LV-algebras and determines the unit group in the LV-domain case. Section~3 proves the Bergman--Preusser embedding result for weighted roses outside the domain cases. Section~4 handles the multivertex case using representation graphs and Sanov's theorem. In Section~5 we combine these results with the domain classification of Hazrat and Preusser to obtain the classification theorem.

The proof separates the one-vertex and multivertex cases. For LV-roses, the local valuation argument given in \cite{HaP} shows that all invertible elements are scalar. For weighted roses outside the domain cases, we use the Bergman--Preusser monoid argument to produce a unital copy of a noncommutative ordinary Leavitt path algebra. The theorem of Hai and Khanh then supplies the required free subgroup in characteristic $0$. When $|E^0|>1$, representation graphs allow us to isolate a two-dimensional subspace on which suitable invertible elements act as the Sanov matrices given in \cite{S}.

We now fix notation. A \emph{weighted graph} $(E,\omega)=(E^0,E^{\mathrm{st}},E^1,s,r,\omega)$ consists of a set of vertices $E^0$, a set of structured edges $E^{\mathrm{st}}$, source and range maps $s,r:E^{\mathrm{st}}\to E^0$, and a weight map $\omega:E^{\mathrm{st}}\to\mathbb N$. Each structured edge $\mathtt{e}\in E^{\mathrm{st}}$ gives real edges $\mathtt{e}_1,\ldots,\mathtt{e}_{\omega(\mathtt{e})}$, and $E^1=\bigsqcup_{\mathtt{e}\in E^{\mathrm{st}}}\{\mathtt{e}_1,\ldots,\mathtt{e}_{\omega(\mathtt{e})}\}$, with $s(\mathtt{e}_i)=s(\mathtt{e})$ and $r(\mathtt{e}_i)=r(\mathtt{e})$. The graph is \emph{row-finite} if $s^{-1}(\mathtt{v})$ is finite for every $\mathtt{v}\in E^0$. For a nonsink vertex $\mathtt{v}$, set $\omega(\mathtt{v})=\max\{\omega(\mathtt{e})\mid \mathtt{e}\in E^{\mathrm{st}},\ s(\mathtt{e})=\mathtt{v}\}$, and for a sink we set $\omega(\mathtt{v})=0$.

Let $(E,\omega)$ be a row-finite weighted graph. If $E^0$ is finite, then $(E,\omega)$ is finite. Indeed, $E^{\mathrm{st}}=\bigcup_{\mathtt{v}\in E^0}s^{-1}(\mathtt{v})$, and each $s^{-1}(\mathtt{v})$ is finite. Hence $E^{\mathrm{st}}$ is finite. Since $\omega(\mathtt{e})<\infty$ for every $\mathtt{e}\in E^{\mathrm{st}}$, the set $E^1=\bigsqcup_{\mathtt{e}\in E^{\mathrm{st}}}\{\mathtt{e}_1,\ldots,\mathtt{e}_{\omega(\mathtt{e})}\}$ is finite. Thus, in the row-finite setting, $|E^0|<\infty$ is equivalent to $(E,\omega)$ being finite.

\begin{definition}[Weighted Leavitt path algebra]
Let $(E,\omega)$ be a weighted graph and $K$ a field. The \emph{weighted Leavitt path algebra} of $(E,\omega)$ over $K$, denoted $L_K(E,\omega)$, is the $K$-algebra generated by
$$
\{\mathtt{v}\mid \mathtt{v}\in E^0\}\cup
\{\mathtt{e}_i,\mathtt{e}_i^*\mid \mathtt{e}\in E^{\mathrm{st}},\ 1\leq i\leq \omega(\mathtt{e})\},
$$
subject to the following relations:
\begin{enumerate}[label=\textup{(W\arabic*)}]
\item $\mathtt{u}\mathtt{v}=\delta_{\mathtt{u}\mathtt{v}}\mathtt{u}$ for all $\mathtt{u},\mathtt{v}\in E^0$;
\item $s(\mathtt{e})\mathtt{e}_i=\mathtt{e}_i=\mathtt{e}_i r(\mathtt{e})$ and $r(\mathtt{e})\mathtt{e}_i^*=\mathtt{e}_i^*=\mathtt{e}_i^*s(\mathtt{e})$, for all $\mathtt{e}\in E^{\mathrm{st}}$ and $1\leq i\leq\omega(\mathtt{e})$;
\item for every nonsink vertex $\mathtt{v}\in E^0$ and all $1\leq i,j\leq\omega(\mathtt{v})$,
$$
\sum_{\mathtt{e}\in s^{-1}(\mathtt{v})}\mathtt{e}_i\mathtt{e}_j^*=\delta_{ij}\mathtt{v};
$$
\item for all $\mathtt{e},\mathtt{f}\in E^{\mathrm{st}}$,
$$
\sum_{i=1}^{\max\{\omega(\mathtt{e}),\omega(\mathtt{f})\}}\mathtt{e}_i^*\mathtt{f}_i
=\delta_{\mathtt{e}\mathtt{f}}r(\mathtt{e}).
$$
\end{enumerate}
In \textup{(W3)} and \textup{(W4)} we use the convention that $\mathtt{e}_i=\mathtt{e}_i^*=0$ whenever $i>\omega(\mathtt{e})$.
\end{definition}

If $\omega(\mathtt{e})=1$ for every $\mathtt{e}\in E^{\mathrm{st}}$, then each structured edge $\mathtt{e}$ gives only one real edge $\mathtt{e}_1$. Identifying $\mathtt{e}$ with $\mathtt{e}_1$, relations \textup{(W1)}--\textup{(W4)} become the usual Cuntz--Krieger relations for the ordinary Leavitt path algebra of the graph with vertex set $E^0$ and edge set $E^{\mathrm{st}}$. Thus, in the unweighted case, $L_K(E,\omega)$ is naturally isomorphic to the ordinary Leavitt path algebra $L_K(E)$.

We also use representation graphs in the sense of \cite{HPS}. Let $\widehat E=(\widehat E^0,\widehat E^1,\widehat s,\widehat r)$ be the directed graph associated to $(E,\omega)$, where $\widehat E^0=E^0$, $\widehat E^1=\{\mathtt{e}_i\mid \mathtt{e}\in E^{\mathrm{st}},\ 1\leq i\leq\omega(\mathtt{e})\}$, and $\widehat s(\mathtt{e}_i)=s(\mathtt{e})$, $\widehat r(\mathtt{e}_i)=r(\mathtt{e})$. For $\mathtt{e}_i\in\widehat E^1$, set $\operatorname{tag}(\mathtt{e}_i)=i$ and $\operatorname{st}(\mathtt{e}_i)=\mathtt{e}$.

\begin{definition}\label{definition_VF}
A \textit{representation graph} of $(E,\omega)$ is a pair $(F,\phi)$, where $F=(F^0,F^1,s_F,r_F)$ is a directed graph and $\phi=(\phi^0,\phi^1):F\to\widehat E$ is a graph homomorphism, such that:
\begin{enumerate}[label=\textup{(\arabic*)}]
\item for any $v\in F^0$ and $1\leq i\leq\omega(\phi^0(v))$, there is precisely one $f\in s_F^{-1}(v)$ with $\operatorname{tag}(\phi^1(f))=i$;
\item for any $v\in F^0$ and any $\mathtt{e}\in E^{\mathrm{st}}$ with $r(\mathtt{e})=\phi^0(v)$, there is precisely one $f\in r_F^{-1}(v)$ with $\operatorname{st}(\phi^1(f))=\mathtt{e}$.
\end{enumerate}
We refer these two conditions as \textit{outgoing condition} and \textit{incoming condition} respectively.
\end{definition}

For a representation graph $(F,\phi)$, let $V_F$ be the $K$-vector space with basis $F^0$. For $\mathtt{u}\in E^0$, $\mathtt e\in E^{\mathrm{st}}$, and $1\leq i\leq\omega(\mathtt e)$, define endomorphisms $\sigma_u,\sigma_{\mathtt e_i},\sigma_{\mathtt e_i^*}\in\operatorname{End}_K(V_F)$ on basis elements $v\in F^0$ by
$$
\begin{aligned}
\sigma_u(v)&=
\begin{cases}
v,&\phi^0(v)=u,\\
0,&\text{otherwise},
\end{cases}
\\[1ex]
\sigma_{\mathtt e_i}(v)&=
\begin{cases}
r_F(f),&\text{if }f\in s_F^{-1}(v)\text{ and }\phi^1(f)=\mathtt e_i,\\
0,&\text{otherwise},
\end{cases}
\\[1ex]
\sigma_{\mathtt e_i^*}(v)&=
\begin{cases}
s_F(f),&\text{if }f\in r_F^{-1}(v)\text{ and }\phi^1(f)=\mathtt e_i,\\
0,&\text{otherwise}.
\end{cases}
\end{aligned}
$$
Then, the universal property of $L_K(E,\omega)$ induces an algebra homomorphism
$$
\pi:L_K(E,\omega)\longrightarrow \operatorname{End}_K(V_F)^{\mathrm{op}},
$$
given by $\pi(u)=\sigma_u$, $\sigma_{\mathtt e_i}$ and $\sigma_{\mathtt e_i^*}$.
In particular, via $\pi$, the vector space $V_F$ becomes a right $L_K(E,\omega)$-module, where the action is given by
$$
a\cdot v := \pi(a)(v),
$$
for all $a\in L_K(E,\omega)$ and $v\in V_F$; see \cite[p.~168]{HPS}. 

Throughout, $K$ denotes a field and all weighted graphs are connected and finite. Hence $L_K(E,\omega)$ is unital with identity $1=\sum_{v\in E^0}v$.

\section{Unit groups of LV-algebras}

In this section we determine the unit group in the LV-domain case. The main tool is the local valuation given in \cite{HaP}. For an LV-rose this local valuation becomes a valuation, and this forces every invertible element to have support contained in $E^0$. Consequently, the only invertible elements are scalar.

\begin{definition}[{\cite[Definition~38]{HaP}}]
A weighted graph $(E,\omega)$ is called an \emph{LV-graph} if it is connected,
$E^{\mathrm{st}}\neq\emptyset$, 
$$
\omega(\mathtt{e})\geq 2 \text{ for every } \mathtt{e}\in E^{\mathrm{st}}, \text{ and } \bigl|\{\mathtt{e}\in s^{-1}(\mathtt{v})\mid
\omega(\mathtt{e})=\omega(\mathtt{v})\}\bigr|\geq 2 \text{ for every non-sink vertex } \mathtt{v},
$$
where $\omega(\mathtt{v})=\max\{\omega(\mathtt{e})\mid
\mathtt{e}\in s^{-1}(\mathtt{v})\}$.
An LV-graph with one vertex is called an \emph{LV-rose}. The weighted Leavitt
path algebras of LV-graphs are called \emph{LV-algebras}.
\end{definition}

\begin{definition}[{\cite[Definition 36]{HaP}}]
Let $(E,\omega)$ be a weighted graph and $R$ a domain. For $a\in L_R(E,\omega)$, the
\emph{support} of $a$, denoted $\operatorname{supp}(a)$, is the set of all normal
generalised paths which occur with nonzero coefficient in the normal form
$\operatorname{NF}(a)$.
\end{definition}

\begin{definition}[{\cite[Definition 37]{HaP}}]\label{def:local valuation}
Let $(E,\omega)$ be a weighted graph and let $R$ be a domain. A \emph{local valuation} on $L_R(E,\omega)$
is a map $\nu:L_R(E,\omega)\to \mathbb N_0\cup\{-\infty\}$ such that
\begin{enumerate}[label=\textup{(\arabic*)}]
    \item $\nu(a)=-\infty$ if and only if $a=0$; $\nu(a)=0$ if and only if $a\neq 0$ and $\operatorname{supp}(a)\subseteq E^0$;
    \item $\nu(a+b)\leq \max\{\nu(a),\nu(b)\}$ for all $a,b\in L_R(E,\omega)$;
    \item $\nu(ab)=\nu(a)+\nu(b)$ whenever $v\in E^0$, $a\in L_R(E,\omega)v$, and $b\in vL_R(E,\omega)$.
\end{enumerate}
We use the conventions
$-\infty\leq x$ and $x+(-\infty)=(-\infty)+x=-\infty$.
\end{definition}

\begin{proposition}[{\cite[Proposition 40]{HaP}}]\label{prop:local-valuation}
Let $(E,\omega)$ be an LV-graph. Then $\nu(a)=\max\{|p|\mid p\in \operatorname{supp}(a)\}$ for $a\in L_R(E,\omega)$,
with $\max(\emptyset)=-\infty$, defines a local valuation on $L_R(E,\omega)$.
\end{proposition}

We shall call $\nu$ a \textit{valuation}
if $\nu(ab)=\nu(a)+\nu(b)$ for all $a,b$.

\begin{theorem}\label{thm:equiv-units-lv-domain}
Let $K$ be a field, let $(E,\omega)$ be an LV-graph, and put
$A=L_K(E,\omega)$. Then the following statements are equivalent:
\begin{enumerate}[label=\textup{(\arabic*)}]
\item $(E,\omega)$ is an LV-rose.
\item The function $\nu$ defined in Proposition \ref{prop:local-valuation} is a valuation on $A$.
\item $A$ is a domain.
\item $A$ is unital and $A^\times \cong K^\times$.
\end{enumerate}
\end{theorem}

\begin{proof}
By \cite[Proposition 40]{HaP}, the map $\nu$ is a local valuation
on $A$.

$(1)\Rightarrow(2)$. Suppose that $E^0=\{\mathtt{v}\}$. Then $\mathtt{v}=1$, and hence $A=A\mathtt{v}=\mathtt{v}A$. It follows from Definition \ref{def:local valuation}(3) that $\nu(ab)=\nu(a)+\nu(b)$ for all $a,b\in A$.

$(2)\Rightarrow(3)$. Suppose that $\nu$ is a valuation. If $a,b\in A$ are
nonzero, then $\nu(a),\nu(b)\in\mathbb N_0$, and so
$\nu(ab)=\nu(a)+\nu(b)\neq -\infty$. Hence $ab\neq 0$. Thus $A$ is a domain.

$(3)\Rightarrow(1)$. This is precisely \cite[Theorem 41]{HaP} in
the case of LV-graphs: since $K$ is a field and $(E,\omega)$ is an LV-graph,
$L_K(E,\omega)$ is a domain only when $(E,\omega)$ is an LV-rose.

$(1)\Rightarrow(4)$. Write $E^0=\{\mathtt{v}\}$. Then $A$ is unital with identity $\mathtt{v}$.
Let $u\in A^\times$, and let $u^{-1}$ be its inverse. By $(2)$, we have
$$
0=\nu(\mathtt{v})=\nu(uu^{-1})=\nu(u)+\nu(u^{-1}).
$$
Since $u$ and $u^{-1}$ are nonzero, both $\nu(u)$ and $\nu(u^{-1})$ lie in
$\mathbb N_0$. Thus $\nu(u)=\nu(u^{-1})=0$. By the definition of $\nu$, this
means that $\operatorname{supp}(u)\subseteq E^0$ and
$\operatorname{supp}(u^{-1})\subseteq E^0$. Hence $u=r\mathtt{v}$ and $u^{-1}=s\mathtt{v}$ for
some nonzero $r,s\in K$. From $uu^{-1}=u^{-1}u=\mathtt{v}$, we get $rs=sr=1$. Thus
$r\in K^\times$, and $u\in K^\times \mathtt{v}$. The reverse inclusion is immediate, since
$(r\mathtt{v})^{-1}=r^{-1}\mathtt{v}$ for $r\in K^\times$. Therefore
$A^\times=K^\times \mathtt{v}=K^\times 1$.

$(4)\Rightarrow(1)$. Suppose by contradiction that $(E,\omega)$ is
not an LV-rose. If $A$ is not unital, then (4) fails. Thus assume that $A$ is
unital. Then $E^0$ is finite and $|E^0|>1$. Since $E$ is connected, there is a
structured edge $\mathtt{e}$ with $s(\mathtt{e})\neq r(\mathtt{e})$. Choose
$1\leq i\leq \omega(\mathtt{e})$ and put $p=\mathtt{e}_i$. Then $p\neq 0$ and
$p=s(\mathtt{e})p r(\mathtt{e})$. Since $s(\mathtt{e})r(\mathtt{e})=0$, we have $p^2=0$.
Therefore $1+p$ is an invertible element, with inverse $1-p$.

This invertible element is not a scalar multiple of $1$. Indeed, if $1+p=r1$ for some
$r\in K^\times$, then multiplying on the left by $s(\mathtt{e})$ and on the right by
$r(\mathtt{e})$ gives $p=0$, a contradiction. Hence
$A^\times\neq K^\times 1$, so (4) fails. Thus the implication $(4)\Rightarrow(1)$ follows.
\end{proof}

\section{Ordinary Leavitt path algebra subalgebras in weighted roses}
In this section we record the Bergman--Preusser construction for weighted roses. It provides a structural explanation for the one-vertex case: outside the domain cases, a weighted rose contains a unital copy of a noncommutative ordinary Leavitt path algebra. The free subgroup result in this case will then follow from \cite{HK}. Let $R$ be an associative ring with identity. Recall that the $V$-monoid, denoted by $V(R)$, consists of the isomorphism classes of finitely generated projective right $R$-modules. This set forms an abelian monoid under the addition operation induced by the direct sum, namely 
$$
[P] + [Q] := [P \oplus Q]
$$
for all $[P], [Q] \in V(R)$.

\begin{theorem}[{\cite[Theorem~6.2]{B}}]\label{thm:Bergman-6.2-corrected}
Let $K$ be a field, and let $M$ be a finitely generated commutative conical monoid with a distinguished element $I\neq 0$. Assume that for every $x\in M$ there exist $y\in M$ and $n>0$ such that $x+y=nI$.
Then there exists a $K$-algebra $R$, right and left hereditary, such that $V(R)\cong M$ as monoids with distinguished element, with $[R]$ corresponding to $I$.

Moreover, $R$ can be chosen to have the following weak universal property: for every $K$-algebra $B$ and every monoid homomorphism
$$
\varphi:M\to V(B)
$$
such that $\varphi(I)=[B]$, there exists a $K$-algebra homomorphism $f:R\to B$ such that the map induced by extension of scalars
$$
-\otimes_R B:V(R)\to V(B)
$$
is precisely $\varphi$ under the identification $V(R)\cong M$. 
\end{theorem}
We next recall Preusser's monoid~\cite{Pre4,Pre1} associated to a weighted graph. 
\begin{definition}\label{defM}
Let $(E,\omega)$ be a weighted graph. For any $\mathtt{v}\in E^0$ write $\omega(s^{-1}(\mathtt{v}))=\{\omega_1(\mathtt{v}),\dots,$ $\omega_{k_{\mathtt{v}}}(\mathtt{v})\}$, where $k_{\mathtt{v}}\geq 0$ and $\omega_1(\mathtt{v})<\dots<\omega_{k_{\mathtt{v}}}(\mathtt{v})$ (hence $k_{\mathtt{v}}$ is the number of different weights of edges in $s^{-1}(\mathtt{v})$). Let $M_{(E,\omega)}$ be the commutative monoid generated by set $\{\mathtt{v},q_1^{\mathtt{v}},\dots,q^{\mathtt{v}}_{k_{\mathtt{v}}-1}\mid \mathtt{v}\in E^0\}$ subject to the relations 
\begin{equation}\label{monrel}
q^{\mathtt{v}}_{i-1}+(\omega_i(\mathtt{v})-\omega_{i-1}(\mathtt{v}))\mathtt{v}=q_i^{\mathtt{v}}+\sum_{\substack{\mathtt{e}\in s^{-1}(\mathtt{v}),\\\omega(\mathtt{e})=\omega_i(\mathtt{v})}}r(\mathtt{e})\quad\quad(\mathtt{v}\in E^0,1\leq i\leq k_{\mathtt{v}})
\end{equation}
where $q^{\mathtt{v}}_0=q^{\mathtt{v}}_{k_{\mathtt{v}}}=\omega_0(\mathtt{v})=0$. 
\end{definition}

It was shown that $V (L_K(E,\omega))) \cong M_{(E,\omega)}$ with $[L_K(E,\omega)]=\sum_{\mathtt{v}\in E^0}\mathtt{v}$ (see~\cite{Pre4,Pre1}). 
\begin{theorem}\label{thm:ordinary-LPA-in-non-LV-rose}
Let $K$ be a field, let $(E,\omega)$ be a weighted rose with unique vertex $\mathtt{v}$. Suppose that $(E,\omega)$ is neither an isolated vertex, nor the rose graph $R_1$, nor an LV-rose. Then there exists a finite ordinary graph $G$ such that $L_K(G)$ is noncommutative and there is a unital embedding
$$
L_K(G)\hookrightarrow L_K(E,\omega).
$$
\end{theorem}

\begin{proof}
Write
$$
\omega(s^{-1}(\mathtt{v}))
=\{\omega_1<\cdots<\omega_k\},
$$
and, for $1\leq j\leq k$, let
$$
n_j=\bigl|\{\mathtt{e}\in E^{\mathrm{st}}
\mid \omega(\mathtt{e})=\omega_j\}\bigr|.
$$
Since $(E,\omega)$ is a rose, the relations of
$M_{(E,\omega)}$ given in (\ref{monrel}) take the form
\begin{equation}\label{eq:rose-monoid-relation}
q^{\mathtt{v}}_{j-1}
+(\omega_j-\omega_{j-1})\mathtt{v}
=
q^{\mathtt{v}}_j+n_j\mathtt{v},
\qquad 1\leq j\leq k,
\end{equation}
where $q^{\mathtt{v}}_0=q^{\mathtt{v}}_k=0$ and $\omega_0=0$.

We first consider the case $k=1$. If $\omega_1=1$, then all weights are
trivial, so $L_K(E,\omega)$ is the ordinary Leavitt path algebra of the rose
with $n_1$ petals. Since $(E,\omega)$ is not $R_1$, we have $n_1\geq2$.
Thus the conclusion holds by taking $G$ to be this ordinary rose.

Suppose now that $k=1$ and $\omega_1=N\geq2$. Since $(E,\omega)$ is not
an LV-rose, there is only one structured loop, say $\mathtt{e}$, and
$\omega(\mathtt{e})=N$. Let $G=R_N$ be the ordinary rose with loops
$c_1,\ldots,c_N$. Then the assignments
$$
c_i\longmapsto \mathtt{e}_i^*,
\qquad
c_i^*\longmapsto \mathtt{e}_i,
\qquad
v_G\longmapsto\mathtt{v},
$$
gives 
$$
L_K(G)\cong L_K(E,\omega).
$$
Since $N\geq2$, it follows that $L_K(G)$ is noncommutative.

It remains to consider $k\geq2$. We first observe that
$$
q^{\mathtt{v}}_j\neq0
\qquad\text{for }1\leq j\leq k-1.
$$
Let $\mathcal B=\{0,\infty\}$ be the commutative monoid with
$0+\infty=\infty+0=\infty$ and $\infty+\infty=\infty$. Assign
$$
\mathtt{v}\longmapsto\infty
\quad\text{and}\quad
q^{\mathtt{v}}_j\longmapsto\infty
\quad (1\leq j\leq k-1).
$$
Since $\omega_j-\omega_{j-1}>0$ and $n_j\geq1$, both sides of every
relation \eqref{eq:rose-monoid-relation} are sent to $\infty$. Thus these
assignments induce a monoid homomorphism $M_{(E,\omega)}\longrightarrow\mathcal B$.
In particular, we have that $q^{\mathtt{v}}_j\ne 0$ for all $1\leq j\leq k-1$. We now obtain a relation
\begin{equation}\label{eq:v-mv-q}
\mathtt{v}=m\mathtt{v}+\mathtt{q}\quad\text{with } m\geq1\text{ and }\mathtt{q}\neq0.
\end{equation}
If $\omega_1=1$, relation
\eqref{eq:rose-monoid-relation} for $j=1$ has the form $\mathtt{v}=n_1\mathtt{v}+q^{\mathtt{v}}_1$,
so we take $m=n_1$ and $\mathtt{q}=q^{\mathtt{v}}_1$. If
$\omega_1\geq2$, then $(E,\omega)$ fails to be an LV-rose only because
there is a unique structured edge of maximal weight. Thus $n_k=1$, and
the relation for $j=k$ becomes
$$
q^{\mathtt{v}}_{k-1}
+(\omega_k-\omega_{k-1})\mathtt{v}
=\mathtt{v}.
$$
In this case take
$$
m=\omega_k-\omega_{k-1} 
\quad\text{and}\quad \mathtt{q}=q^{\mathtt{v}}_{k-1}.
$$
Let $G$ be the following ordinary graph:
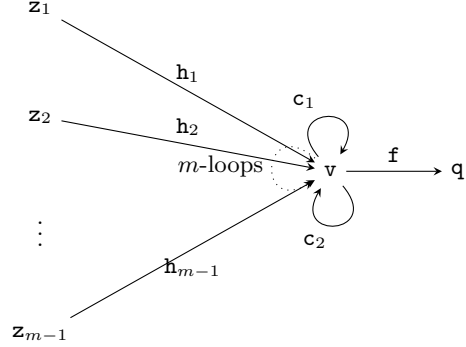
\begin{figure}[H]
\centering
\begin{tikzpicture}[>=stealth, every node/.style={font=\small}, scale=1.2]

\node (z1) at (-3.2,1.8) {$\mathtt{z}_1$};
\node (z2) at (-3.2,0.6) {$\mathtt{z}_2$};
\node (zd) at (-3.2,-0.6) {$\vdots$};
\node (zm) at (-3.2,-1.8) {$\mathtt{z}_{m-1}$};
\node (v)  at (0,0) {$\mathtt{v}$};
\node (q)  at (1.4,0) {$\mathtt{q}$};

\draw[->] (z1) -- node[above] {$\mathtt{h}_1$} (v);
\draw[->] (z2) -- node[above] {$\mathtt{h}_2$} (v);
\draw[->] (zm) -- node[below] {$\mathtt{h}_{m-1}$} (v);

\draw[->] (v) edge[out=130,in=60,loop] node[above left] {$\mathtt{c}_1$} ();
\draw[->] (v) edge[out=-50,in=-120,loop] node[below left] {$\mathtt{c}_2$} ();
\draw[dotted,->] (v) edge[out=140,in=210,loop] node[left] {$m$-loops} ();

\draw[->] (v) -- node[above] {$\mathtt{f}$} (q);

\end{tikzpicture}
\end{figure}
When $m=1$, the list
$z_1,\ldots,z_{m-1}$ is empty. Its graph monoid has the presentation
$$
M_G=
\left\langle
z_1,\ldots,z_{m-1},v_G,q_G
\ \middle|\
z_i=v_G,\quad v_G=mv_G+q_G
\right\rangle .
$$
Under the identification $V(L_K(G))\cong M_G$, the distinguished element is
$$
\relax[L_K(G)]=z_1+\cdots+z_{m-1}+v_G+q_G=mv_G+q_G=v_G.
$$
Define
$$
\overline{\phi}:M_G\longrightarrow M_{(E,\omega)}
$$
by
$$
\overline{\phi}(z_i)=\mathtt{v},
\qquad
\overline{\phi}(v_G)=\mathtt{v},
\qquad
\overline{\phi}(q_G)=\mathtt{q}.
$$
Relation \eqref{eq:v-mv-q} shows that $\overline{\phi}$ is well-defined.
Moreover,
$$
\overline{\phi}([L_K(G)])
=\mathtt{v}
=[L_K(E,\omega)].
$$
Thus $\overline{\phi}$ is a homomorphism of monoids with distinguished
element. By Theorem \ref{thm:Bergman-6.2-corrected}, there is a unital $K$-algebra homomorphism
$$
\phi:L_K(G)\longrightarrow L_K(E,\omega)
$$
which induces $\overline{\phi}$ on $V$-monoids. The graph $G$ satisfies Condition \textup{(L)}, thus the Cuntz--Krieger uniqueness
theorem therefore implies that $\phi$ is injective; see
\cite[Theorem~2.2.16]{AAS}. Hence $L_K(G)$ embeds unitally into
$L_K(E,\omega)$. Also it is clear that $L_K(G)$ is noncommutative. This completes the proof.
\end{proof}

\begin{corollary}\label{cor:non-lv-rose-free-subgroup}
Let $K$ be a field of characteristic $0$, let $(E,\omega)$ be a weighted graph with
$|E^0|=1$. Suppose that $(E,\omega)$ is not an
LV-rose. If $(E,\omega)$ is neither an isolated vertex nor the rose graph $R_1$, then $L_K(E,\omega)^\times$ contains noncyclic free subgroups.
\end{corollary}
\begin{proof}
It follows from Theorem~\ref{thm:ordinary-LPA-in-non-LV-rose} that there exists a graph $F$ such that $L_K(F)$ is non-commutative and embedded as a subring into $L_K(E,\omega)$. By \cite[Theorem 3.4]{HK}, we conclude that $L_K(F)^\times$ contain a non-cyclic free subgroup, so does $L_K(E,\omega)^\times$.
\end{proof}

The preceding theorem settles the question for weighted roses outside the domain cases. There is another basic class for which the answer is also affirmative. If $(E,\omega)$ satisfies \textup{(LPA)} conditions, then \cite[Theorem~5.2.1]{Pre4} identifies $L_K(E,\omega)$ with an ordinary Leavitt path algebra $L_K(F)$ for a suitable graph $F$. Hence, under the assumption that $L_K(E,\omega)$ is noncommutative, the algebra $L_K(E,\omega)$ itself is a unital copy of a noncommutative ordinary Leavitt path algebra. This suggests the following question.

\begin{question}
Let $(E,\omega)$ be a finite connected weighted graph such that $L_K(E,\omega)$ is noncommutative. Is it true that $L_K(E,\omega)$ contains a unital copy of a noncommutative ordinary Leavitt path algebra if and only if $(E,\omega)$ is not an LV-rose?
\end{question}

\section{Free subgroups in weighted Leavitt path algebras}

In this section we complete the proof of the free-subgroup part of the paper. The case $|E^0|=1$ follows from Corollary~\ref{cor:non-lv-rose-free-subgroup}. It remains to consider the case $|E^0|>1$. Since the graph is connected, there is a structured edge whose source and range are distinct. We shall use representation graphs to lift such an edge and then obtain the two Sanov's matrices from the action on a two-dimensional subspace.

\begin{lemma}\label{lem:prescribed-lift}
Let $(E,\omega)$ be a weighted graph, and let $\widehat E$ be the directed graph associated to $(E,\omega)$. Let $\mathtt{e}_i\in \widehat E^1$, where $\mathtt{e}\in E^{\mathrm{st}}$ and $1\leq i\leq \omega(\mathtt{e})$. Then there is a connected representation graph $(F,\phi)$ of $(E,\omega)$ containing an edge $e_i$ such that $\phi^1(e_i)=\mathtt{e}_i$.
\end{lemma}

\begin{proof}
Put $\mathtt{u}=s(\mathtt{e})$ and $\mathtt{v}=r(\mathtt{e})$. We begin with the smallest possible lift of the edge $\mathtt{e}_i$. Let $F_0$ be the directed graph with two vertices $u,v$ and one edge $e_i$ from $u$ to $v$. Define
$$
\phi_0^0(u)=\mathtt{u},\qquad
\phi_0^0(v)=\mathtt{v},\quad\text{and}\quad
\phi_0^1(e_i)=\mathtt{e}_i.
$$
Then $\phi_0:F_0\to \widehat E$ is a graph homomorphism.

The graph $F_0$ need not satisfy the two conditions in Definition~\ref{definition_VF}. We will enlarge it step by step. At each stage we add exactly the missing edges required by the outgoing and incoming conditions, while keeping uniqueness. That means we construct an increasing sequence of graphs inductively:
$$
(F_0,\phi_0)\subseteq (F_1,\phi_1)\subseteq (F_2,\phi_2)\subseteq\cdots .
$$

Assume that $(F_n,\phi_n)$ has been constructed. For $z\in F_n^0$ and $j\in\mathbb N$, define
$$
O_n(z,j)=
\{h\in F_n^1\mid s_{F_n}(h)=z,\ 
\operatorname{tag}(\phi_n^1(h))=j\}.
$$
Thus $O_n(z,j)$ records the edges in $F_n$ which start at $z$ and have tag $j$. Similarly, for $z\in F_n^0$ and $\mathtt{f}\in E^{\mathrm{st}}$ with $r(\mathtt{f})=\phi_n^0(z)$, define
$$
I_n(z,\mathtt{f})=
\{h\in F_n^1\mid r_{F_n}(h)=z,\ 
\operatorname{st}(\phi_n^1(h))=\mathtt{f}\}.
$$
Thus $I_n(z,\mathtt{f})$ records the edges in $F_n$ which end at $z$ and have structure edge $\mathtt{f}$. We shall make sure that, for every $n$,
\begin{equation}\label{eq:In-On}
|O_n(z,j)|\leq 1\qquad \text{and}\qquad |I_n(z,\mathtt{f})|\leq 1.
\end{equation}
These inequalities hold for $n=0$.

We now identify the defects at stage $n$. The outgoing defects are
$$
D_n^+=\{(z,j)\mid z\in F_n^0,\ 1\leq j\leq \omega(\phi_n^0(z)),
\ O_n(z,j)=\emptyset\}.
$$
The incoming defects are
$$
D_n^-=\{(z,\mathtt{f})\mid z\in F_n^0,\ \mathtt{f}\in E^{\mathrm{st}},
\ r(\mathtt{f})=\phi_n^0(z),\ I_n(z,\mathtt{f})=\emptyset\}.
$$
So $D_n^+$ lists the missing outgoing tags, while $D_n^-$ lists the missing incoming structure edges.

We first repair the outgoing defects. Let $(z,j)\in D_n^+$. Since $1\leq j\leq \omega(\phi_n^0(z))$, there exists a structured edge $\mathtt{f}(z,j)\in E^{\mathrm{st}}$ such that
$$
s(\mathtt{f}(z,j))=\phi_n^0(z),\quad\text{where}\quad
\omega(\mathtt{f}(z,j))\geq j.
$$
Add a new vertex $z_j^+$ and a new edge $f_{z,j}^+$ from $z$ to $z_j^+$. Define
$$
\phi_{n+1}^0(z_j^+)=r(\mathtt{f}(z,j))\quad\text{and}\quad
\phi_{n+1}^1(f_{z,j}^+)=\mathtt{f}(z,j)_j.
$$

We next repair the incoming defects. Let $(z,\mathtt{f})\in D_n^-$. Add a new vertex $z_{\mathtt{f}}^-$ and a new edge $f_{z,\mathtt{f}}^-$ from $z_{\mathtt{f}}^-$ to $z$. Define
$$
\phi_{n+1}^0(z_{\mathtt{f}}^-)=s(\mathtt{f})\quad\text{and}\quad
\phi_{n+1}^1(f_{z,\mathtt{f}}^-)=\mathtt{f}_1.
$$
All new vertices and new edges are taken to be distinct. On the old graph $F_n$, set $\phi_{n+1}=\phi_n$.

We now check that $\phi_{n+1}:F_{n+1}\to\widehat E$ is a graph homomorphism. For an edge $f_{z,j}^+$ added to repair an outgoing defect, we have
$$
\widehat s(\mathtt{f}(z,j)_j)=s(\mathtt{f}(z,j))=\phi_n^0(z)
$$
and
$$
\widehat r(\mathtt{f}(z,j)_j)=r(\mathtt{f}(z,j))=\phi_{n+1}^0(z_j^+).
$$
For an edge $f_{z,\mathtt{f}}^-$ added to repair an incoming defect, we have
$$
\widehat s(\mathtt{f}_1)=s(\mathtt{f})=\phi_{n+1}^0(z_{\mathtt{f}}^-)
$$
and
$$
\widehat r(\mathtt{f}_1)=r(\mathtt{f})=\phi_n^0(z).
$$
Thus the source and range are respected for every new edge, and hence $\phi_{n+1}$ is a graph homomorphism.

We now verify that the inequalities \eqref{eq:In-On} are preserved. Let $z\in F_n^0$ be an old vertex. For each missing outgoing tag $j$, we add exactly one edge starting at $z$ with tag $j$, and we add no edge for a tag already represented in $O_n(z,j)$. Hence $|O_{n+1}(z,j)|\leq 1$. Similarly, for each missing incoming structure edge $\mathtt{f}$, we add exactly one edge ending at $z$ with structure edge $\mathtt{f}$, and we add no edge for a structure edge already represented in $I_n(z,\mathtt{f})$. Hence $|I_{n+1}(z,\mathtt{f})|\leq 1$.

It remains to check the new vertices. First suppose that $z'=z_j^+$. Then $\phi_{n+1}^0(z')=r(\mathtt{f}(z,j))$, and the only edge of $F_{n+1}$ ending at $z'$ is $f_{z,j}^+$. Hence, for any $\mathtt{g}\in E^{\mathrm{st}}$ with $r(\mathtt{g})=\phi_{n+1}^0(z')$,
$$
I_{n+1}(z',\mathtt{g})=
\begin{cases}
\{f_{z,j}^+\},&\text{if }\mathtt{g}=\mathtt{f}(z,j),\\
\emptyset,&\text{if }\mathtt{g}\neq \mathtt{f}(z,j).
\end{cases}
$$
No edge of $F_{n+1}$ starts at $z'$, so
$$
O_{n+1}(z',\ell)=\emptyset
$$
for $1\leq \ell\leq \omega(\phi_{n+1}^0(z'))$. Thus the inequalities \eqref{eq:In-On} hold at $z'$.

Now suppose that $z'=z_{\mathtt{f}}^-$. Then $\phi_{n+1}^0(z')=s(\mathtt{f})$, and the only edge of $F_{n+1}$ starting at $z'$ is $f_{z,\mathtt{f}}^-$. Its image is $\mathtt{f}_1$. Hence, for $1\leq \ell\leq \omega(\phi_{n+1}^0(z'))$,
$$
O_{n+1}(z',\ell)=
\begin{cases}
\{f_{z,\mathtt{f}}^-\},&\text{if }\ell=1,\\
\emptyset,&\text{if }\ell\neq 1.
\end{cases}
$$
No edge of $F_{n+1}$ ends at $z'$, so $I_{n+1}(z',\mathtt{g})=\emptyset$ for every $\mathtt{g}\in E^{\mathrm{st}}$ with $r(\mathtt{g})=\phi_{n+1}^0(z')$. Thus the inequalities \eqref{eq:In-On} also hold at $z'$.

This completes the induction. Put
$$
F=\bigcup_{n\geq 0}F_n\qquad\text{and}\qquad
\phi=\bigcup_{n\geq 0}\phi_n.
$$
Since every new vertex is attached by an edge to a vertex already present, the graph $F$ is connected. The original edge $e_i$ from $u$ to $v$ remains in $F$, and $\phi^1(e_i)=\mathtt{e}_i$.

It remains only to check that $(F,\phi)$ is a representation graph. Let $z\in F^0$, and choose $n$ such that $z\in F_n^0$. If $1\leq j\leq \omega(\phi^0(z))$, then either $O_n(z,j)\neq\emptyset$, or $(z,j)\in D_n^+$. In the second case, the edge $f_{z,j}^+$ is added in $F_{n+1}$. Thus there is at least one edge $h\in s_F^{-1}(z)$ with $\operatorname{tag}(\phi^1(h))=j$. There is at most one such edge because $|O_m(z,j)|\leq 1$ for every finite $m$, and any two edges of $F$ already occur together in some $F_m$.

Similarly, let $\mathtt{f}\in E^{\mathrm{st}}$ satisfy $r(\mathtt{f})=\phi^0(z)$. Then either $I_n(z,\mathtt{f})\neq\emptyset$, or $(z,\mathtt{f})\in D_n^-$. In the second case, the edge $f_{z,\mathtt{f}}^-$ is added in $F_{n+1}$. Thus there is at least one edge $h\in r_F^{-1}(z)$ with $\operatorname{st}(\phi^1(h))=\mathtt{f}$. There is at most one such edge by the same finite-stage argument.

Therefore $(F,\phi)$ satisfies both the outgoing and incoming conditions in Definition~\ref{definition_VF}. Hence $(F,\phi)$ is a connected representation graph of $(E,\omega)$ containing the prescribed lift $e_i$ of $\mathtt{e}_i$.
\end{proof}

We now illustrate the construction in two examples, beginning with the case in which the procedure stops after one step.

\begin{example}\label{ex:prescribed-lift-two-vertices}
Let $(E,\omega)$ be the weighted graph with $E^0=\{\mathtt{u},\mathtt{v}\}$ and $E^{\mathrm{st}}=\{\mathtt{e}\}$ such that $s(\mathtt{e})=\mathtt{u}$, $r(\mathtt{e})=\mathtt{v}$ and $\omega(\mathtt{e})=n$. Then $\widehat E^0=\{\mathtt{u},\mathtt{v}\}$ and $\widehat E^1=\{\mathtt{e}_1,\ldots,\mathtt{e}_n\}$, and for each $i\in\{1,\dots,n\}$, we have $s(\mathtt{e}_i)=\mathtt{u}$ and $r(\mathtt{e}_i)=\mathtt{v}$. The representation graph $(F,\phi)$ is given by
$$
F^0=\{u,v_1,\ldots,v_n\}\quad\text{and}\quad F^1=\{e_1,\ldots,e_n\},
$$
where, for each $1\leq i\leq n$, we define $s(e_i)=u$ and $r(e_i)=v_i$.
Define
$$
\phi^0(u)=\mathtt{u},\qquad \phi^0(v_i)=\mathtt{v},\qquad
\phi^1(e_i)=\mathtt{e}_i,\qquad 1\leq i\leq n.
$$
Then Definition~\ref{definition_VF} is clearly satisfied, and so $(F,\phi)$ is a connected representation graph of $E$. 

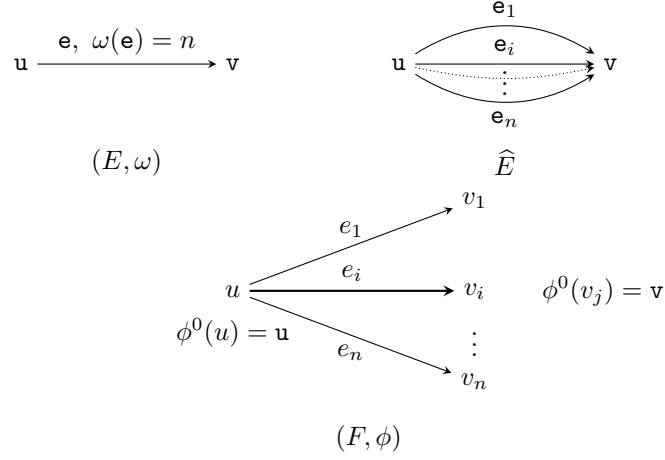
\begin{figure}[H]
\centering
\begin{tikzpicture}[>=stealth, node distance=2.5cm]

\node (Eu) at (0,0) {$\mathtt{u}$};
\node (Ev) at (2.8,0) {$\mathtt{v}$};
\draw[->] (Eu) -- node[above] {$\mathtt{e},\ \omega(\mathtt{e})=n$} (Ev);
\node at (1.4,-1.3) {$(E,\omega)$};

\node (Hu) at (5,0) {$\mathtt{u}$};
\node (Hv) at (7.8,0) {$\mathtt{v}$};
\draw[->] (Hu) to[bend left=32] node[above] {$\mathtt{e}_1$} (Hv);
\draw[->, densely dotted] (Hu) to[bend right=12] (Hv);
\node at (6.4,-0.15) {$\vdots$};
\draw[->] (Hu) -- node[above] {$\mathtt{e}_i$} (Hv);
\draw[->] (Hu) to[bend right=32] node[below] {$\mathtt{e}_n$} (Hv);
\node at (6.4,-1.3) {$\widehat E$};

\node (u) at (2.8,-3.0) {$u$};
\node (v1) at (6.0,-1.8) {$v_1$};
\node (vi) at (6.0,-3.0) {$v_i$};
\node (vn) at (6.0,-4.2) {$v_n$};
\node (dots) at (6.0,-3.6) {$\vdots$};

\draw[->] (u) -- node[above] {$e_1$} (v1);
\draw[->, thick] (u) -- node[above] {$e_i$} (vi);
\draw[->] (u) -- node[below] {$e_n$} (vn);

\node at (2.8,-3.55) {$\phi^0(u)=\mathtt{u}$};
\node at (7.7,-3.0) {$\phi^0(v_j)=\mathtt{v}$};
\node at (4.6,-4.95) {$(F,\phi)$};

\end{tikzpicture}
\end{figure}
In terms of the sets $O_n(z,j)$ and $I_n(z,\mathtt{f})$ in the proof of
Lemma~\ref{lem:prescribed-lift}, start with $F_0$ consisting of the two
vertices $u,v_i$ and the edge $e_i$ from $u$ to $v_i$. Thus
$$
\phi_0^0(u)=\mathtt{u},\qquad \phi_0^0(v_i)=\mathtt{v},\qquad
\phi_0^1(e_i)=\mathtt{e}_i.
$$
At the vertex $u$, we have
$$
O_0(u,i)=\{e_i\}\quad\text{and}\quad
O_0(u,j)=\emptyset\quad (1\leq j\leq n,\ j\neq i).
$$
Since $\mathtt{v}$ is a sink, there is no outgoing condition at $v_i$.
For the incoming condition, no structured edge of $E$ ends at $\mathtt{u}$,
while the only structured edge ending at $\mathtt{v}$ is $\mathtt{e}$. Hence
$$
I_0(v_i,\mathtt{e})=\{e_i\}.
$$
Therefore
$$
D_0^+=\{(u,j)\mid 1\leq j\leq n,\ j\neq i\}\quad\text{and}\quad D_0^-=\emptyset.
$$
The first step adds, for every $j\neq i$, a new vertex $v_j$ and an edge $e_j$ from $u$ to $v_j$ such that
$$
\phi_1^0(v_j)=\mathtt{v}\quad\text{and}\quad \phi_1^1(e_j)=\mathtt{e}_j.
$$
Thus, in $F_1$, one has
$$
O_1(u,j)=\{e_j\}\qquad (1\leq j\leq n),
$$
and, for each $1\leq j\leq n$,
$$
I_1(v_j,\mathtt{e})=\{e_j\}.
$$
Again, there is no outgoing condition at any $v_j$, and no incoming condition
at $u$. Hence
$$
D_1^+=D_1^-=\emptyset.
$$
Thus the construction stops after the first step, and $F_1=F$.
\end{example}

\begin{example}\label{ex:prescribed-lift-both-defects}
Let $(E,\omega)$ be the weighted graph with $E^0=\{\mathtt{u},\mathtt{v}\}$ and $E^{\mathrm{st}}=\{\mathtt{e},\mathtt{f}\}$ such that 
$$
s(\mathtt{e})=r(\mathtt{f})=\mathtt{u}\quad\text{and}\quad
r(\mathtt{e})=s(\mathtt{f})=\mathtt{v},
$$
and $\omega(\mathtt{e})=2$ and $\omega(\mathtt{f})=1$.
Thus $\widehat E^1=\{\mathtt{e}_1,\mathtt{e}_2,\mathtt{f}_1\}$. The representation graph $(F,\phi)$ is constructed as follows. We start with the graph $F_0$ having two vertices $u,v$ and one edge $e_1$ from $u$ to $v$, with
$$
\phi_0^0(u)=\mathtt{u},\qquad
\phi_0^0(v)=\mathtt{v},\qquad
\phi_0^1(e_1)=\mathtt{e}_1.
$$
We now compute the sets $O_0$ and $I_0$. In the notations used in the below figure, the initial graph $F_0$ has vertices $u_0,v_1$ and one edge $e_1$ from $u_0$ to $v_1$, with
$$
\phi_0^0(u_0)=\mathtt{u},\qquad
\phi_0^0(v_1)=\mathtt{v},\qquad
\phi_0^1(e_1)=\mathtt{e}_1.
$$
Thus
$$
O_0(u_0,1)=\{e_1\},\qquad O_0(u_0,2)=\emptyset,\qquad
O_0(v_1,1)=\emptyset.
$$
For the incoming condition, the only structured edge ending at $\mathtt{v}$
is $\mathtt{e}$, while the only structured edge ending at $\mathtt{u}$ is
$\mathtt{f}$. Hence
$$
I_0(v_1,\mathtt{e})=\{e_1\}\quad\text{and}\quad I_0(u_0,\mathtt{f})=\emptyset.
$$
Therefore
$$
D_0^+=\{(u_0,2),(v_1,1)\}\quad\text{and}\quad
D_0^-=\{(u_0,\mathtt{f})\}.
$$

Both defect sets are nonempty. To pass from $F_0$ to $F_1$, we add:
\begin{itemize}
\item[$\cdot$] an edge labeled $e_2$ from $u_0$ to $v_2$, with image
$\mathtt{e}_2$ and $\phi_1^0(v_2)=\mathtt{v}$, corresponding to
$(u_0,2)\in D_0^+$;
\item[$\cdot$] an edge labeled $f_1$ from $v_1$ to $u_1$, with image
$\mathtt{f}_1$ and $\phi_1^0(u_1)=\mathtt{u}$, corresponding to
$(v_1,1)\in D_0^+$;
\item[$\cdot$] an edge labeled $f_1$ from $v_{-1}$ to $u_0$, with image
$\mathtt{f}_1$ and $\phi_1^0(v_{-1})=\mathtt{v}$, corresponding to
$(u_0,\mathtt{f})\in D_0^-$.
\end{itemize}
Thus $F_1$ has the form
$$
v_{-1}\xrightarrow{f_1}u_0,\qquad
u_0\xrightarrow{e_1}v_1,\qquad
u_0\xrightarrow{e_2}v_2,\qquad
v_1\xrightarrow{f_1}u_1.
$$

At this stage the relevant sets are
$$
O_1(u_0,1)=\{e_1\},\qquad O_1(u_0,2)=\{e_2\},
$$
$$
O_1(v_{-1},1)=\{f_1\},\qquad O_1(v_1,1)=\{f_1\},
$$
while
$$
O_1(v_2,1)=\emptyset,\qquad
O_1(u_1,1)=O_1(u_1,2)=\emptyset.
$$
For incoming edges one has
$$
I_1(u_0,\mathtt{f})=\{f_1\},\qquad
I_1(u_1,\mathtt{f})=\{f_1\},
$$
$$
I_1(v_1,\mathtt{e})=\{e_1\},\qquad
I_1(v_2,\mathtt{e})=\{e_2\},\qquad
I_1(v_{-1},\mathtt{e})=\emptyset.
$$
Hence $D_1^+$ and $D_1^-$ are still nonempty. Iterating this procedure
yields a connected representation graph $F=\bigcup_{n\geq 0}F_n$. 
Pictorially, we have

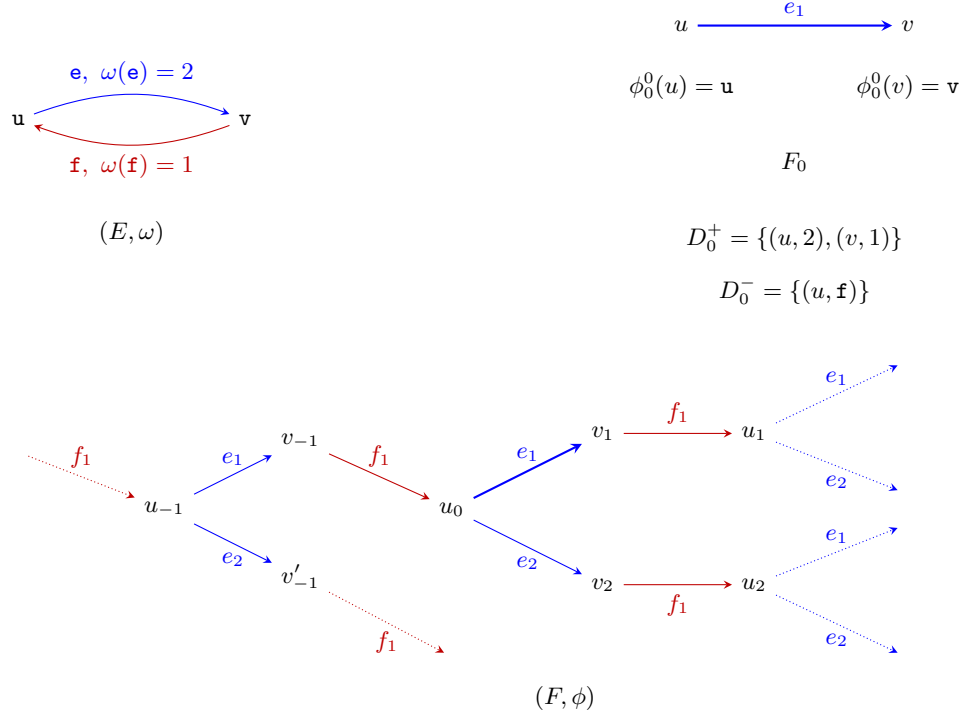
\begin{figure}[H]
\centering

\begin{minipage}{0.42\textwidth}
\centering
\begin{tikzpicture}[>=stealth, every node/.style={font=\small}]

\node (u) at (0,0) {$\mathtt{u}$};
\node (v) at (3,0) {$\mathtt{v}$};

\draw[->,blue] (u) to[bend left=20]
node[above] {$\mathtt{e},\ \omega(\mathtt{e})=2$} (v);
\draw[->,red!75!black] (v) to[bend left=20]
node[below] {$\mathtt{f},\ \omega(\mathtt{f})=1$} (u);

\node at (1.5,-1.5) {$(E,\omega)$};
\end{tikzpicture}
\end{minipage}
\hfill
\begin{minipage}{0.42\textwidth}
\centering
\begin{tikzpicture}[>=stealth, every node/.style={font=\small}]

\node (u) at (0,0) {$u$};
\node (v) at (3,0) {$v$};

\draw[->,blue,thick] (u) -- node[above] {$e_1$} (v);

\node at (0,-0.8) {$\phi_0^0(u)=\mathtt{u}$};
\node at (3,-0.8) {$\phi_0^0(v)=\mathtt{v}$};
\node at (1.5,-1.8) {$F_0$};
\node at (1.5,-2.8) {$D_0^+=\{(u,2),(v,1)\}$};
\node at (1.5,-3.5) {$D_0^-=\{(u,\mathtt{f})\}$};

\end{tikzpicture}
\end{minipage}

\vspace{1.8em}

\begin{minipage}{0.96\textwidth}
\centering
\begin{tikzpicture}[>=stealth, every node/.style={font=\small}]

\node (um1)  at (-3.8,0)   {$u_{-1}$};
\node (vm1)  at (-2.0,0.9) {$v_{-1}$};
\node (vpm1) at (-2.0,-0.9) {$v_{-1}'$};
\node (u0)   at (0,0)      {$u_0$};
\node (v1)   at (2.0,1.0)  {$v_1$};
\node (v2)   at (2.0,-1.0) {$v_2$};
\node (u1)   at (4.0,1.0)  {$u_1$};
\node (u2)   at (4.0,-1.0) {$u_2$};

\draw[->,blue] (um1) -- node[above] {$e_1$} (vm1);
\draw[->,blue] (um1) -- node[below] {$e_2$} (vpm1);
\draw[->,red!75!black] (vm1) -- node[above] {$f_1$} (u0);

\draw[->,blue,thick] (u0) -- node[above] {$e_1$} (v1);
\draw[->,blue] (u0) -- node[below] {$e_2$} (v2);
\draw[->,red!75!black] (v1) -- node[above] {$f_1$} (u1);
\draw[->,red!75!black] (v2) -- node[below] {$f_1$} (u2);

\draw[->,red!75!black,densely dotted] (-5.6,0.7) -- node[above] {$f_1$} (um1);
\draw[->,red!75!black,densely dotted] (vpm1) -- node[below] {$f_1$} (-0.1,-1.9);

\draw[->,blue,densely dotted] (u1) -- node[above] {$e_1$} (5.9,1.9);
\draw[->,blue,densely dotted] (u1) -- node[below] {$e_2$} (5.9,0.25);

\draw[->,blue,densely dotted] (u2) -- node[above] {$e_1$} (5.9,-0.25);
\draw[->,blue,densely dotted] (u2) -- node[below] {$e_2$} (5.9,-1.9);

\node at (1.5,-2.5) {$(F,\phi)$};
\end{tikzpicture}
\end{minipage}

\label{fig:prescribed-lift-both-defects}
\end{figure}
\end{example}

Next, we will use the following classical result of Sanov. In \cite{S},
Sanov proved that the subgroup of $\operatorname{GL}_2(\mathbb Z)$ generated by
$$
U=\begin{pmatrix}1&2\\0&1\end{pmatrix}\quad\text{and}\quad
V=\begin{pmatrix}1&0\\2&1\end{pmatrix}
$$
is a free group of rank two (see also \cite[Example~25, p.~26]{dH}).

\begin{proposition}\label{prop:nonloop-edge-free-subgroup}
Let $K$ be a field of characteristic $0$, let $(E,\omega)$ be a
weighted graph. Suppose that there is a structured
edge $\mathtt{f}\in E^{\mathrm{st}}$ such that
$s(\mathtt{f})\neq r(\mathtt{f})$. Then $L_K(E,\omega)^\times$ contains a non-cyclic
free subgroup.
\end{proposition}

\begin{proof}
Put $A=L_K(E,\omega)$. Choose $1\leq i\leq\omega(\mathtt{f})$, and put $\mathtt{v}_0=s(\mathtt{f})$, $\mathtt{v}_1=r(\mathtt{f})$ and $a=\mathtt{f}_i$.
Thus $\mathtt{v}_0\neq \mathtt{v}_1$, and we get $a^2=(a^*)^2=0$. It
follows that
$$
U=1+2a\quad\text{and}\quad W=1+2a^*
$$
are in $A^\times$, with $U^{-1}=1-2a$ and $W^{-1}=1-2a^*$ respectively.

By Lemma~\ref{lem:prescribed-lift}, there is a connected representation graph
$(F,\phi)$ of $E$ and an edge $f_i$ from $v_0$ to $v_1$ in $F$ such that $\phi^1(f_i)=\mathtt{f}_i$. Then $\phi^0(v_0)=\mathtt{v}_0$ and $\phi^0(v_1)=\mathtt{v}_1$.

Let $V_F$ be the right $A$-module associated to $(F,\phi)$, with basis
$F^0$. 
Since $s(f_i)=v_0$ and $r(f_i)=v_1$ and $\phi^1(f_i)=\mathtt{f}_i$, it is straightforward to check that 
$$
v_0\cdot a=v_1,\qquad v_1\cdot a^*=v_0, \qquad v_1\cdot a=0, \quad\text{ and }\quad v_0\cdot a^*=0.
$$

Set $\mathcal W=Kv_0\oplus Kv_1$, and for
$\xi=\lambda v_0+\mu v_1\in\mathcal W$, write $[\xi]_{(v_0,v_1)}=(\lambda,\mu)$
for its row coordinate vector with respect to the ordered basis $(v_0,v_1)$.
For $r\in A$ such that $\mathcal W \cdot r \subseteq \mathcal W$, let $M_r\in \mathrm{M}_2(K)$ be
the matrix of the right action of $r$ on $\mathcal W$, determined by
$$
[\xi\cdot r]_{(v_0,v_1)}=[\xi]_{(v_0,v_1)}M_r,\quad\text{where}\quad\xi\in\mathcal W.
$$
Since $v_0\cdot U=v_0+2v_1$ and $v_1\cdot U=v_1$, we have
$$
(1,0)M_U=(1,2)\qquad\text{and}\qquad (0,1)M_U=(0,1).
$$
Hence
$$
M_U=
\begin{pmatrix}
1&2\\
0&1
\end{pmatrix}.
$$
Similarly,
$$
M_W=
\begin{pmatrix}
1&0\\
2&1
\end{pmatrix}.
$$
Thus, with respect to row coordinates on the ordered basis $(v_0,v_1)$, the
restrictions of the right actions of $U$ and $W$ to $\mathcal W$ are presented by the following matrices:
$$
M_U=
\begin{pmatrix}
1&2\\
0&1
\end{pmatrix}
\quad\text{and}\quad
M_W=
\begin{pmatrix}
1&0\\
2&1
\end{pmatrix}.
$$

Since $\operatorname{char}K=0$, Sanov's theorem \cite{S} applies
to these matrices: the subgroup
$\langle M_U,M_W\rangle\leq \operatorname{GL}_2(K)$ is free of rank two.
Let $F_2=\langle X,Y\rangle$ be the free group of rank two on $X$ and $Y$.
Since $U,W\in A^\times$, the universal property of $F_2$ gives a
unique group homomorphism $\theta:F_2\longrightarrow A^\times$ such that $\theta(X)=U$ and $\theta(Y)=W$.
Thus, if $w(X,Y)\in F_2$ is a reduced word in $X$ and $Y$, then $\theta(w(X,Y))=w(U,W)$.
In particular, we have $\theta(1_{F_2})=1$.

Let $w(X,Y)\in F_2$ be nontrivial reduced word in $X$ and $Y$. As $M_U$ and $M_W$ are the generators of a free group of rank 2, we get that $w(M_U,M_W)\neq I_2$. On the other hand, the action of $w(U,W)$ on $\mathcal W$, with respect to
row coordinates in the ordered basis $(v_0,v_1)$, is given by the matrix
$w(M_U,M_W)$. Hence $w(U,W)\neq 1$, since $1$ acts on
$\mathcal W$ as $I_2$. Therefore no nontrivial reduced word lies in
$\ker\theta$, and so $\ker\theta=\{1\}$. Hence
$\langle U,W\rangle$ is a free group of rank two in $A^\times$.
\end{proof}

\begin{corollary}\label{cor:multivertex-free-subgroup}
Let $K$ be a field of characteristic $0$, let $(E,\omega)$ be a
weighted graph. If $|E^0|>1$, then $L_K(E,\omega)^\times$
contains a non-cyclic free subgroup.
\end{corollary}

\begin{remark}\label{rem:positive-characteristic}
The assumption $\operatorname{char}K=0$ in
Proposition~\ref{prop:nonloop-edge-free-subgroup} is used only to invoke
Sanov's theorem. There is an analogous statement in positive characteristic.
Indeed, if $\operatorname{char}K=p>0$ and $K$ contains two elements
$\lambda,\mu$ algebraically independent over the field having $p$ elements, then
\cite[Lemma~2.8, p.~30, and Exercises~2.2--2.3, p.~31]{W}
implies that
$$
C=\begin{pmatrix}\mu&0\\ \lambda&\mu^{-1}\end{pmatrix}
\quad\text{and}\quad
D=\begin{pmatrix}\mu&\lambda\\ 0&\mu^{-1}\end{pmatrix}
$$
generate a free subgroup of rank two in $\operatorname{GL}_2(K)$.
Thus the same proof applies after replacing the Sanov matrices by the following
ones. In the notation of Proposition~\ref{prop:nonloop-edge-free-subgroup},
where $a=\mathtt{f}_i$, $\mathtt{v}_0=s(\mathtt{f})$, and
$\mathtt{v}_1=r(\mathtt{f})$, put
$$
U_{\lambda,\mu}
=
1-\mathtt{v}_0-\mathtt{v}_1
+\mu\mathtt{v}_0+\lambda a+\mu^{-1}\mathtt{v}_1,
$$
and
$$
W_{\lambda,\mu}
=
1-\mathtt{v}_0-\mathtt{v}_1
+\mu\mathtt{v}_0+\lambda a^*+\mu^{-1}\mathtt{v}_1.
$$
A direct calculation shows that these elements are units, with
$$
U_{\lambda,\mu}^{-1}
=
1-\mathtt{v}_0-\mathtt{v}_1
+\mu^{-1}\mathtt{v}_0-\lambda a+\mu\mathtt{v}_1,
$$
and
$$
W_{\lambda,\mu}^{-1}
=
1-\mathtt{v}_0-\mathtt{v}_1
+\mu^{-1}\mathtt{v}_0-\lambda a^*+\mu\mathtt{v}_1.
$$
On the space $Kv_0\oplus Kv_1$, these elements act, respectively,
as $D$ and $C$. Hence they generate a free subgroup of rank two. Consequently,
the conclusions of Proposition~\ref{prop:nonloop-edge-free-subgroup}
remain valid over such
fields of positive characteristic.
\end{remark}

\section{Weighted Leavitt path algebras with abelian unit groups}

We now prove the classification theorem. The domain cases are supplied by the classification given in \cite{HaP}. Outside these cases, it was proved in the preceding sections that, in characteristic $0$, non-cyclic free subgroups always exist in $L_K(E,\omega)^\times$.

\begin{theorem}\label{thm:abelian-units-and-free-subgroups}
Let $K$ be a field with $\mathrm{char} K =0 $, let $(E,\omega)$ be a weighted graph,
and put $A=L_K(E,\omega)$. Then the following statements are equivalent:
\begin{enumerate}[label=\textup{(\arabic*)}]
\item $A^\times$ is abelian.
\item $A$ is a domain.
\item One of the following holds:
\begin{enumerate}[label=\textup{(\alph*)}]
\item $E^0=\{\mathtt{v}\}$ and $E^{\mathrm{st}}=\emptyset$;
\item $E^0=\{\mathtt{v}\}$, $E^{\mathrm{st}}=\{\mathtt{e}\}$, and
$\omega(\mathtt{e})=1$;
\item $(E,\omega)$ is an LV-rose.
\end{enumerate}
\item One of the following holds:
\begin{enumerate}[label=\textup{(\alph*)}]
\item $A=K\mathtt{v}$, and so $A^\times\cong K^\times$;
\item $A\cong K[x,x^{-1}]$, and so
$A^\times=K^\times\langle x\rangle$, where $\langle x\rangle$ denotes the cyclic subgroup of $K[x,x^{-1}]^\times$ generated by $x$;
\item $E^0=\{\mathtt{v}\}$, $(E,\omega)$ is an LV-rose, and
$A^\times\cong K^\times$.
\end{enumerate}
\item $A^\times$ contains no non-cyclic free subgroup.
\end{enumerate}
\end{theorem}

\begin{proof}
The equivalence $(2)\Leftrightarrow(3)$ is exactly \cite[Theorem~41]{HaP}.
Thus $A$ is a domain precisely in the three cases listed in $(3)$.

We prove $(3)\Rightarrow(4)$. In case $(3)(a)$, $A=K\mathtt{v}$, hence
$A^\times=K^\times\mathtt{v}$. In case $(3)(b)$, relations (W3) and (W4) implies that $\mathtt{e}_1\mathtt{e}_1^*=\mathtt{v}=\mathtt{e}_1^*\mathtt{e}_1$.
Hence the mapping $x\mapsto \mathtt{e}_1$ induces an isomorphism
$K[x,x^{-1}]\cong A$, and $A^\times=K^\times\mathtt{e}_1^{\mathbb Z}$.
In case $(3)(c)$, by Theorem~\ref{thm:equiv-units-lv-domain},
$A^\times=K^\times\mathtt{v}$. Hence $(3)\Rightarrow(4)$. Since the
groups in $(4)$ are abelian, $(4)\Rightarrow(1)$.

It remains to prove $(1)\Rightarrow(3)$. 
Assume by contradiction that $(E,\omega)$ is not one of the three graphs in $(3)$.
First suppose that $|E^0|>1$. Since $E$ is connected, there is
$\mathtt{f}\in E^{\mathrm{st}}$ with
$s(\mathtt{f})\neq r(\mathtt{f})$. Choose
$1\leq i\leq\omega(\mathtt{f})$, and put
$$
a=\mathtt{f}_i,\qquad
\mathtt{v}_0=s(\mathtt{f}),\qquad\text{and}\qquad
\mathtt{v}_1=r(\mathtt{f}).
$$
Then $\mathtt{v}_0\neq\mathtt{v}_1$, and so $a^2=0=(a^*)^2$. Hence
$$
U=1+a\quad\text{and}\quad W=1+a^*
$$
are invertible in $A$, with inverses $1-a$ and $1-a^*$.

By Lemma~\ref{lem:prescribed-lift}, there is a representation graph
$(F,\phi)$ containing an edge $f_i$ from $v_0$ to $v_1$ such that
$\phi^1(f_i)=\mathtt{f}_i$. Let $V_F$ be the right
$A$-module associated to $(F,\phi)$. Since
$$
\widehat s(\mathtt{f}_i)=\mathtt{v}_0\quad\text{and}\quad
\widehat r(\mathtt{f}_i)=\mathtt{v}_1,
$$
the action on $V_F$ implies that
$$
v_0\cdot a=v_1,\qquad v_1\cdot a^*=v_0,\qquad
v_1\cdot a=0,\qquad v_0\cdot a^*=0.
$$
Consequently
$$
v_0\cdot aa^*=(v_0\cdot a)\cdot a^*=v_1\cdot a^*=v_0 \qquad \text{whereas} \qquad v_0\cdot a^*a=(v_0\cdot a^*)\cdot a=0.
$$
Thus $aa^*\neq a^*a$ and so $UW\neq WU$, which implies that $A^\times$ is not abelian.

Now suppose that $|E^0|=1$. Since $(E,\omega)$ is not one of the graphs in $(3)$, it is neither an isolated vertex, nor $R_1$, nor an LV-rose. By Corollary~\ref{cor:non-lv-rose-free-subgroup}, $A^\times$ contains a non-cyclic free subgroup. In particular, $A^\times$ is not abelian. This
proves that $(1)\Rightarrow(3)$, and hence $(1)$--$(4)$ are equivalent.

If the equivalent conditions
$(1)$--$(4)$ hold, then $A^\times$ is abelian, so it contains no
non-cyclic free subgroup. Conversely, suppose that these conditions fail.
If $|E^0|>1$, then there is $\mathtt{f}\in E^{\mathrm{st}}$ with
$s(\mathtt{f})\neq r(\mathtt{f})$, and
Proposition~\ref{prop:nonloop-edge-free-subgroup} shows that there is a free subgroup of rank two
in $A^\times$. If $|E^0|=1$, then
Corollary~\ref{cor:non-lv-rose-free-subgroup} yields a free subgroup of rank two inside $A^\times$. This proves the equivalence of $(1)$--$(5)$.
\end{proof}

\section*{Acknowledgements}
The author is deeply grateful to Roozbeh Hazrat for suggesting the question which led to this paper, and for many helpful discussions on weighted Leavitt path algebras. His comments and suggestions substantially improved the manuscript, especially the part concerning the Bergman--Preusser monoid argument and the embedding of ordinary Leavitt path algebras into weighted Leavitt path algebras. The author also thanks him for his encouragement and for carefully reading earlier versions of the paper. 

\section*{Funding}
The author is supported by the Vietnam National Foundation for Science and Technology Development (NAFOSTED) under Grant No.~101.04-2025.41.

\end{document}